\documentclass[11pt]{amsart}

\usepackage[colorlinks=true, linkcolor=blue, anchorcolor=blue, citecolor=blue, filecolor=blue, menucolor= blue, urlcolor=red]{hyperref}

\usepackage{amsmath,amssymb,amsthm}

\long\def\eatit#1{}
\newtheorem{thm}{Theorem}[subsection]
\newtheorem{prop}[thm]{Proposition}
\newtheorem{lem}[thm]{Lemma}
\newtheorem{cor}[thm]{Corollary}
\newtheorem{fact}[thm]{Fact}

\newtheorem{defn}[thm]{Definition}

\newtheorem{example}[thm]{Example}
\newtheorem{conj}[thm]{Conjecture}

\newtheorem{ques}[thm]{Question}
\newtheorem{rem}[thm]{Remark}

\newcommand{\pr}[1]{{{\bf P}^{#1}}}

\begin{document}

\title{Are symbolic powers highly evolved?}

\author{Brian Harbourne \& Craig Huneke}

\address{Brian Harbourne\\
Department of Mathematics\\
University of Nebraska\\
Lincoln, NE 68588-0130, USA}
\email{bharbour@math.unl.edu}

\address{Craig Huneke\\ 
Department of Mathematics\\
University of Kansas\\
Lawrence, KS 66045-7523, USA}
\email{huneke@math.ukans.edu}

\date{September 8, 2011}

\thanks{The first author's work on this project was sponsored by the National Security Agency under Grant/Cooperative agreement ``Advances on Fat Points and Symbolic Powers,'' Number H98230-11-1-0139. The United States Government is authorized to reproduce and distribute reprints notwithstanding any copyright notice. The second author gratefully acknowledges the support of the NSF. We also thank R. Lazarsfeld for bringing the paper \cite{refEV} to our attention and C. Bocci, for helpful discussions.}

\begin{abstract} Searching for structural reasons behind old results 
and conjectures of Chudnovksy regarding the
least degree of a nonzero form in an ideal of fat points in $\pr N$, we 
make conjectures which explain them, and we prove the conjectures 
in certain cases, including the case of general 
points in $\pr 2$. Our conjectures were also partly motivated by 
the Eisenbud-Mazur Conjecture on 
evolutions, which concerns
symbolic squares of prime ideals in local rings, but in contrast we consider
higher symbolic powers of homogeneous ideals in polynomial rings.
\end{abstract}

 %Projective techniques


\maketitle

\renewcommand{\thethm}{\thesection.\arabic{thm}}
\setcounter{thm}{0}

\section{Introduction}\label{section: i}

Both authors of this paper have been interested for many years
 in the behavior of symbolic powers of ideals
in regular rings, especially in polynomial rings. In the case of defining ideals of points,
symbolic powers are special cases of ideals of so-called fat points, and their study provides a meeting
ground of geometry and algebra. Hereafter $R$ will denote the ring $K[x_0,\ldots,x_N]=K[\pr N]$
where $K$ is a field and $I\subseteq R$ will be a homogeneous ideal,
with $M=(x_0,\ldots,x_N)$ being the maximal homogeneous ideal in $R$.
Many of our arguments hold for arbitrary fields, but in some cases the field $K$ must be infinite.
By symbolic power we mean $I^{(m)}=R\cap(\cap_P (I^m)_P)$
where the intersections take place in the field of fractions of
$K[\pr N]$, and the second intersection is over all
associated primes $P$ of $I$. An important special case
is that of ideals of fat points;
i.e., $I=\cap_i I(p_i)^{m_i}$ for non-negative integers $m_i$
and a finite set of distinct
points $p_i\in\pr N$, where $I(p_i)$ is the ideal
generated by all forms that vanish at $p_i$. In this case $I^{(m)}$
is just $\cap_i I(p_i)^{mm_i}$. When $\pr N$ is clear, we denote the subscheme
defined by $I$ by $Z=m_1p_1+\cdots+m_np_n$, and denote its ideal $I$ by $I(Z)$.
In this case, $I^{(m)}$ becomes $I(mZ)$.
 
If $J$ is a homogeneous ideal, we let  $\alpha(J)$ be
the least degree of a polynomial in $J$.
Let $I$ be the radical ideal of a finite set of points in $\pr N$.
Using complex analytic techniques, Waldschmidt and Skoda \cite{refW, refSk} showed that
$$\alpha(I^{(m)})/m\ge \alpha(I)/N\eqno{({}^\circ)}$$ 
for every $m>0$.
Interestingly, this result also follows
for any homogeneous ideal $I\subseteq R = K[\pr N]$ because 
$I^{(Nm)}\subseteq I^m$ by \cite{refELS}, \cite{refHH}
and clearly $(I^{(m)})^N\subseteq I^{(mN)}$.
These containments imply that
$N\alpha(I^{(m)})=\alpha((I^{(m)})^N)\ge\alpha(I^{(Nm)})\ge\alpha(I^m)=m\alpha(I)$, so
$\alpha(I^{(m)})/m\ge\alpha(I^{(Nm)})/(Nm)\ge\alpha(I)/N$. 

When $N=2$, Chudnovsky \cite{refCh} improved the bound of Waldschmidt and Skoda \cite{refW, refSk}.
Since only a sketch of Chudnovksy's proof is given in \cite{refCh}, we later give a 
proof (see Proposition~\ref{Chudprop}). 
Chudnovsky's improvement  is the following: Let $p_1,\ldots,p_n\in\pr2$ be distinct points.
Let $I=\cap_i I(p_i)\subset K[\pr2]$.
Then
$\alpha(I^{(m)})/m\ge (\alpha(I)+1)/2$
for all $m>0$.

Because the result of Waldschmidt and Skoda can be explained by a general property of
symbolic powers, that
$I^{(Nm)}\subseteq I^m$ for all $m$, it is natural to speculate whether or not there is
a similar property which might underlie Chudnovsky's improved bound. This led us 
to our first conjecture, which gives a structural reason for the result of
Chudnovsky. Namely, we conjecture that $I^{(2r)}\subseteq M^rI^r$ for an ideal $I$ of points
in $\pr 2$. This conjecture, which we prove for general points (see Proposition~\ref{genptsprop}), easily implies the result of Chudnovsky. We generalize this
conjecture
in a natural way to arbitrary dimension.
 As it turns out, a positive answer to our more general conjecture also gives a positive answer to
a conjecture of Chudnovsky, that if $I$ is the ideal of a finite set of points in $K[\pr N]$, then $\alpha(I^{(m)})/m\geq (\alpha(I) + N -1)/N$.

Our conjecture relates to evolutions.
Evolutions are certain kinds of ring homomorphisms that arose 
in proving Fermat's Last Theorem \cite{refF, refTW, refWi}; see
\cite{refB} for an exposition. 
An important step in the proof was to show in certain cases
only trivial evolutions occurred. Eisenbud and Mazur \cite{refEM}
showed the question of triviality (which for the work 
of Wiles was in mixed characteristic) could be translated into
a statement  involving symbolic powers. They then made the following conjecture
in characteristic $0$: 

\begin{conj}[Eisenbud-Mazur]\label{EMconj}
Let $P\subset {\bf C}[[x_1,\ldots,x_d]]$ be a prime ideal. 
Then $P^{(2)}\subseteq MP$, where $M=(x_1,\ldots,x_d)$. 
\end{conj}

Our main conjecture can at least heuristically be thought of as a generalization of
the conjecture of Eisenbud-Mazur to higher symbolic powers.
The homogeneous version of Conjecture \ref{EMconj} 
for symbolic squares is easy to verify:

\begin{fact}\label{homogEMfact}
Let $I\subseteq K[x_0,\ldots,x_N]$ be a proper
homogeneous ideal where ${\rm char}(K)=0$. 
Then $I^{(2)}\subseteq MI$. 
\end{fact}

\begin{proof} For any $F\in I^{(2)}$ we have 
$\partial F/\partial x_i\in I$; if ${\rm char}(K)=0$, then 
by the Euler identity we have
${\rm deg}(F)F=\sum_ix_i\partial F/\partial x_i\in MI$, 
so $I^{(2)}\subseteq MI$.
\end{proof}

The general question we wish to raise is:

\begin{ques}\label{higherEvoQues}
Let $I\subset R$ be a homogeneous ideal. 
For which $m,i$ and $j$ do we have $I^{(m)}\subseteq M^jI^i$? 
\end{ques}

A complete answer will typically depend on $I$, but it is also of 
interest to ask what holds for all $I$, knowing only $N$.
Since $M^jI^i\subseteq I^i$, we see that 
whenever $I^{(m)}\subseteq M^jI^i$ is true we also have $I^{(m)}\subseteq I^i$.
It is known that $I^{(m)}\subseteq I^i$ holds whenever
$m/i\ge N$ \cite{refELS}, also \cite{refHH}, and that
whenever $m/i<N$, there exist ideals $I$ for which
$I^{(mt)}\subseteq I^{it}$ fails for $t\gg0$ \cite{refBH}. 

Thus an interesting starting point is:

\begin{ques}\label{higherEvoQues2}
Let $I\subset R$ be a homogeneous ideal. 
For which $j$ does $I^{(rN)}\subseteq M^jI^r$
hold for all $I$ and all $r$? 
\end{ques}

The best known general results concerning this question are found in \cite{refHH1} and \cite{refTY}.

In Section 2 we state our main conjecture, and prove it in a very special case. Section 3 
gives the proof of the result of Chudnovsky, which we then slightly generalize and relate
to the Noetherian property for symbolic power algebras. We prove our conjecture for general
points in $\pr 2$.  The last section relates our work to a conjecture of the first author,
and has further speculations and examples. 

\bigskip

\renewcommand{\thethm}{\thesection.\arabic{thm}}
\setcounter{thm}{0}

\section{An optimistic conjecture}\label{section: optconj}

\medskip

Note that $I^{(rN)}\subseteq M^jI^r$ fails in
general if $j>r(N-1)$, even for $I = M$. So the best we can hope for is:

\begin{conj}\label{fatptconj}
Let $I=\cap_i I(p_i)^{m_i}\subset K[\pr N]$ be any fat points ideal. 
Then $I^{(rN)}\subseteq M^{r(N-1)}I^r$
holds for all $r>0$. 
\end{conj}

In fact, as far we know, there is no reason not to raise this question for arbitrary
homogeneous ideals in a polynomial ring, or even for arbitrary ideals in a regular local
ring. In this paper, all our main arguments are for the case of points, so we have chosen
not to make the general conjecture unless more evidence can be found that supports it. 

In fact, Proposition \ref{fatptprop} shows that the conjecture is true for $N=2$ for fat point ideals arising as
symbolic powers of radical ideals generated in a single degree.
Examples of point sets whose ideals are generated in a single 
degree include star configurations (see Definition \ref{starDef})
and any set of $\binom{s}{2}$ general points in $\pr2$.

Most of our progress on this conjecture relates to the minimal degrees of elements in symbolic powers. We need
the following definition. 

\begin{defn}  Let $R$ be a polynomial ring over a field, and let $J$ be a homogeneous
ideal. We set $\alpha(J)$ equal to the smallest integer $l$ such that $J_l\ne 0$, and set $\beta(J)$ equal to the smallest integer $n$ such that
$J_n$ contains a regular sequence of length two.
\end{defn}

We first note that the conjecture holds provided we know some information about the symbolic powers.

\begin{prop}\label{fatptprop1}
Let $J=\cap_i I(p_i)\subset R = K[\pr N]$ and let $I=J^{(m)}\subset R$ be a fat points ideal. 
If for some $s$, 
$J$ is generated by a set of homogeneous elements each having degree at most $s$, 
and if $\alpha(I^{(Nr)})\ge rms+rm(N-1)$, then $I^{(Nr)}\subseteq M^{r(N-1)}I^r$.
\end{prop}

\begin{proof}  We use that $J^{(Nmr)}\subseteq J^{mr}$ by \cite{refELS} (also see \cite{refHH}).
Since $J$ is generated in degree $s$ and less, so $J^{mr}$ is generated in degree
at most $mrs$, and hence $J^{mr}_t=(M^{mr(N-1)}J^{mr})_t$ for $t\ge mrs+mr(N-1)$.
Of course, $0=J^{(Nmr)}_t\subseteq(M^{mr(N-1)}J^{mr})_t$ for $t<\alpha(J^{(Nmr)})$,
while for $t\ge \alpha(J^{(Nmr)})=\alpha(I^{(Nr)})\ge rms+rm(N-1)$
we have $(J^{(Nmr)})_t\subseteq (J^{mr})_t=(M^{mr(N-1)}J^{mr})_t$
and so $I^{(Nr)}=J^{(Nmr)}\subseteq M^{mr(N-1)}J^{mr}\subseteq M^{r(N-1)}I^r$.
\end{proof}

\renewcommand{\thethm}{\thesection.\arabic{thm}}
\setcounter{thm}{0}

\section{Points in projective space}

As discussed in the introduction, if $I$ is the radical ideal of a finite set of points in $\pr N$,
Waldschmidt and Skoda \cite{refW, refSk} showed that
$\alpha(I^{(m)})/m\ge \alpha(I)/N$ for every $m>0$, using
complex analysis, but it also follows 
for any homogeneous ideal $I\subseteq R = K[\pr N]$ using the result
$I^{(Nm)}\subseteq I^m$ of \cite{refELS}, \cite{refHH}. Among other things,
the algebraic argument suggests that 
Conjecture~\ref{fatptconj} is closely related to the study of $\alpha(I^{(m)})$.

When $N=2$, Chudnovsky improved the bound of Waldschmidt and Skoda \cite{refW, refSk}.
Since only a sketch of Chudnovksy's proof is given in \cite{refCh}, we give both the 
statement and proof here. We will give a small improvement of this result in
Proposition \ref{Chudgen}.

\begin{prop}[Chudnovsky]\label{Chudprop}
Let $p_1,\ldots,p_n\in\pr2$ be distinct points.
Let $I=\cap_i I(p_i)\subset K[\pr2]$.
Then 
$$\frac{\alpha(I^{(m)})}{m}\ge \frac{\alpha(I)+1}{2}$$
for all $m>0$. 
\end{prop}

\begin{proof} Let $b=\alpha(I)$. Choose distinct points $q_1,\dots,q_t\in\{p_1,\ldots,p_n\}$
with $t$ as small as possible such that $\alpha(J)=b$,
where $J=\cap_i I(q_i)$. By minimality, the points $q_i$ impose independent conditions in degree
$b-1$, hence $t=\binom{b+1}{2}$ (since this is the dimension of the space of all forms of degree $b-1$)
and $\alpha(J)={\rm reg}(J)$.
Thus $J$ is generated in degree $b$ and hence the only base points of
$J_b$ are the points $q_i$; in particular, $J_b$ is fixed component free
(i.e., there is no nonconstant common factor for the homogeneous elements of $J$ of degree $b$).
Now let $A$ be a nonzero form in $I^{(m)}_a$, where
$a=\alpha(I^{(m)})$. Since $J_b$ is 
fixed component free, we can choose an element $B\in J_b$
with no factor in common with $A$. By B\'ezout's Theorem,
it follows that $ab=\deg(A)\deg(B)\ge mt=m\binom{b+1}{2}$, and hence
that $\alpha(I^{(m)})=a\ge (b+1)/2=(\alpha(I)+1)/2$.
\end{proof}

\eatit{
\begin{cor}\label{Chudcor}
Let $p_1,\ldots,p_n\in\pr2$ be distinct points.
Let $I=\cap_i I(p_i)^m\subset K[\pr2]$.
Then $\alpha(I^{(r)})/r\ge (\alpha(I)+m)/2$
holds for all $m>0$. 
\end{cor}

\begin{proof} Let $J=\cap_i I(p_i)$ so $I=J^{(m)}$ and $I^{(r)}=J^{(rm)}$.
Then $\alpha(I^{(r)})=\alpha(J^{(rm)})\ge rm(\alpha(J)+1)/2=r(\alpha(J^m)+m)/2
\ge r(\alpha(J^{(m)})+m)/2$ by Proposition \ref{Chudprop}. 
\end{proof}
}%end eatit

Our original motivation for Conjecture~\ref{fatptconj} was exactly this result of
Chudnovsky. Just as the containment result $I^{(rN)}\subseteq I^r$ of \cite{refHH}
implies the bound $\alpha(I^{(m)})/m\geq \alpha(I)/N$ of
Waldschmidt and Skoda, we looked for a new containment which in a similar way would
imply the bound in Proposition \ref{Chudprop}. In addition to the result of
Proposition \ref{Chudprop}, Chudnovsky \cite{refCh} has conjectured for $N>2$  
that $\alpha(I^{(m)})/m\ge (\alpha(I)+N-1)/N$ (actually his conjecture was stated
for $K={\bf C}$ for points in affine $N$-space). We now show that Conjecture
\ref{fatptconj} implies not only Proposition \ref{Chudprop} but also $\alpha(I^{(m)})/m\ge (\alpha(I)+N-1)/N$.

\begin{lem}\label{ContainmentImpliesChud}
Let $p_1,\ldots,p_n\in\pr N$ be distinct points, and let $I=\cap_i I(p_i)\subset K[\pr N]$.
If $I^{(Nr)}\subseteq M^{r(N-1)}I^r$ holds for all $r\geq 1$, then
$\alpha(I^{(m)})/m\ge (\alpha(I)+N-1)/N$ holds for all $m\geq1$.
\end{lem}

We pause to recall a numerical quantity introduced by Waldschmidt \cite{refW}
for sets of points, but which extends to homogeneous ideals $0\neq I\subset K[\pr N]$.
Define
$$\gamma(I)=\lim_{m\to\infty}\frac{\alpha(I^{(m)})}{m}.$$
The limit exists and satisfies $\gamma(I)\leq\frac{\alpha(I^{(m)})}{m}$ for all $m\geq 1$;
see \cite[Lemma 2.3.1]{refBH} and its proof.

We now prove Lemma \ref{ContainmentImpliesChud}.

\begin{proof}
Since $I^{(Nr)}\subseteq M^{r(N-1)}I^r$, we have 
$\alpha(I^{(Nr)})\geq \alpha(M^{r(N-1)}I^r)=r\alpha(I)+r(N-1)$.
Now divide by $rN$ and take limits as $r\to\infty$ to get
$\alpha(I^{(m)})/m\geq\gamma(I)=(\alpha(I)+N-1)/N$.
\end{proof}

Conversely, we can also use Proposition \ref{Chudprop} to 
prove certain cases of Conjecture \ref{fatptconj}.

\begin{prop}\label{fatptprop}
Let $J=\cap_i I(p_i)\subset K[\pr N]$ and let $I=J^{(m)}\subset R$ be a fat points ideal.
If $N=2$ and if $J$ is generated in degree $\alpha(J)$,
then $I^{(Nr)}\subseteq M^{r(N-1)}I^r$ for all $r$.
\end{prop}

\begin{proof}
Let $s=\alpha(J)$; then  
$\alpha(J^{(Nmr)})\ge mrs+mr(N-1)$ by Proposition \ref{Chudprop}. 
Now apply Proposition \ref{fatptprop1}.
\end{proof}

\begin{rem}\label{ChudConj}\rm 
Given $I=\cap_i I(p_i)$ for any distinct points $p_1,\ldots,p_n\in\pr N$ for $N>2$,
if Chudnovsky's conjecture \cite{refCh} $\alpha(I^{(m)})/m\ge (\alpha(I)+N-1)/N$ holds, 
then the proof of Proposition \ref{fatptprop} would work 
for any $N$, not just $N=2$.
\end{rem}

We next refine Proposition \ref{Chudprop} by bringing into play 
$\beta(I^{(m)})$.

\begin{prop}\label{Chudgen} Let Let $p_1,\ldots,p_n\in\pr2$ be distinct points.
Let $I=\cap_i I(p_i)\subset K[\pr2] = R$. Set $\alpha_m = \alpha(I^{(m)})$,
and $\beta_m = \beta(I^{(m)})$. Then we have:
\begin{enumerate}
\item[(i)] $\alpha_m\beta_m \geq m^2n$, and
\vskip\baselineskip
\item[(ii)] $\displaystyle\frac{\alpha_m}{m}\ge \Big(\frac{\alpha_1+1}{2}\Big)\Big(\frac{m\alpha_1}{\beta_m}\Big).$
\end{enumerate}
Moreover, if $\alpha_m\beta_m  =  m^2n$, then  
$(I^{(m)})^k = I^{(mk)}$ for all $k\geq 1$. 
\end{prop}

\begin{proof} We first prove that $\alpha_m\beta_m \geq m^2n$, which is basically by B\'ezout's theorem. Let $l$ be a general linear form. Choose
$f$ of degree $\alpha_m$ and $g$ of degree $\beta_m$ in $I^{(m)}$ which form a regular sequence. Then
$\alpha_m\beta_m  = \ell(R/(f,g,l)) = \sum_{P} \ell(_P/(f,g)_P)\ell(R/(P,l))$, where the sum is over
all prime ideals minimal over the ideal $(f,g)$. Since such $P$ include all the ideals corresponding to
the points $p_1,\ldots,p_n$, we can restrict the sum to $P_1,\ldots,P_n$, where $P_i = I(p_i)$. In this case,
$\ell((R/(P_i,l)) = 1$ for each $i$, and $\ell(R_{P_i}/(f,g)_{P_i}) \geq m^2$ since $R_{P_i}$ is a regular local
ring and by assumption, the images of $f$ and $g$ are in $(P_i^m)_{P_i}$\footnote{In general, if $(R,M)$ is a regular local ring
and if $x_1,\ldots,x_d\in M^n$ form a maximal regular sequence, then $\ell(R/(x_1,\ldots,x_d))\geq n^d$. This can be seen, for
example, as follows. If $I$ denotes the integral closure of the ideal $(x_1,\ldots,x_d)$, then $I\subset M^n$ as
the latter ideal is integrally closed. Hence the multiplicity of $I$, which is $\ell(R/(x_1,\ldots,x_d))$, is at least the
multiplicity of $M^n$, which is $n^d$.} . This gives the first inequality.

To prove the second inequality, we use the idea from the proof of Proposition~\ref{Chudprop}. As in that
proof,
 choose distinct points $q_1,\dots,q_t\in\{p_1,\ldots,p_n\}$
with $t$ as small as possible such that $\alpha(J)=\alpha(I)$,
where $J=\cap_i I(q_i)$. As above, 
$t=\binom{\alpha(I)+1}{2}$, and  
$\alpha(J)={\rm reg}(J)$.
Thus $J$ is generated in degree $\alpha(I)$. Note that $I^{(m)}\subset J^{(m)}$, and so 
$\alpha_m\beta_m\geq \alpha(J^{(m))})\beta(J^{(m)})\geq m^2\binom{\alpha(I)+1}{2}$, using the
first inequality of this theorem. This proves the second inequality. 

Finally we prove the last part. Suppose that  $\alpha_m\beta_m = m^2n$. Choose $f$ and $g$ as in the first
part of this theorem. We use the ideas of  Theorem 3.1 of \cite{Hu}, which cannot be used directly since
it deals with the local case of a prime ideal. 

To show $I^{(mk)} = (I^{(m)})^k$ for all $k\geq 1$, since $(f,g)^{k-1}I^{(m)}\subseteq(I^{(m)})^k \subseteq I^{(mk)}$, 
it's enough to show $I^{(mk)} = (f,g)^{k-1}I^{(m)}$. To prove this note that
$(f,g)^{k-1}I^{(m)}\subset I^{(mk)}$, so it suffices to prove the equality locally at each associated
prime of $(f,g)^{k-1}I^{(m)}$. We prove that the associated primes of this ideal are exactly the primes
ideals $P_i = I_{p_i}$. Clearly each of these are associated since they are minimal over the ideal.
To prove they are all the associated primes, we use induction on $k$. If $k = 0$, this is clear by the
definition of symbolic powers. For $k\geq 1$, $(f,g)^{k-1}/(f,g)^{k-1}I^{(m)}\cong (f,g)^{k-1}/(f,g)^k\otimes R/I^{(m)}$
since $(f,g)^k\subset (f,g)^{k-1}I^{(m)}$, and the tensor product is isomorphic to a free $R/I^{(m)}$-module as
$(f,g)^{k-1}/(f,g)^k$ is a free $R/(f,g)$-module because $f,g$ form a regular sequence. Hence the only associated
primes of $(f,g)^{k-1}/(f,g)^{k-1}I^{(m)}$ are $P_1,\ldots,P_n$. Now the exact sequence 
$$0\rightarrow (f,g)^{k-1}/(f,g)^{k-1}I^{(m)}\rightarrow R/(f,g)^{k-1}I^{(m)}\rightarrow R/(f,g)^{k-1}\rightarrow 0$$
shows that the associated primes of $R/(f,g)^{k-1}I^{(m)}$ are contained in the union of the associated primes of
$(f,g)^{k-1}/(f,g)^{k-1}I^{(m)}$ together with the associated primes of $R/(f,g)^{k-1}$. Since the associated primes of $R/(f,g)^{k-1}$
are exactly those of $R/(f,g)$, to finish the proof of our claim we need to prove that all the associated primes of $(f,g)$ are
$P_1,\ldots,P_n$. As this ideal is unmixed, this is equivalent to proving that the only points both $f$ and $g$ vanish at are
$P_1,\ldots,P_n$. 
Recall from above that $\alpha_m\beta_m  = \ell(R/(f,g,l)) = \sum_{P} \ell(R_P/(f,g)_P)\ell(R/(P,l))$, where the sum is over
all prime ideals minimal over the ideal $(f,g)$. Since we are assuming that $\alpha_m\beta_m = m^2n$, and since
$\sum_{P_i}\ell(R_{P_i}/(f,g)_{P_i})\ell(R/(P_i,l))\geq m^2n$, we see that the only primes minimal over $(f,g)$ are $P_1,\ldots,P_n$. 

To finish the proof, we need to prove that $(I^{(mk)})_{P_i} = ((f,g)^{k-1}I^{(m)})_{P_i}$ for every $i$. We know from the
fact equality holds that $\ell(R_{P_i}/(f,g)_{P_i}) = m^2$. However the multiplicity of $(P_i)_{P_i}^m$ is exactly $m^2$,
and since $(f,g)_{P_i}\subset (P_i)_{P_i}^m$, it follows from Rees's theorem \cite[Theorem 11.3.1]{refSH} that  $(f,g)_{P_i}$ is a minimal reduction of $(P_i^{m})_{P_i}$. 
Moreover, $(I^{(m)})_{P_i} = (P_i^{m})_{P_i}$ is integrally closed. By the result of Lipman and Teissier \cite{LT}, the result follows.
\end{proof}

\begin{rem}\label{Chudimprove}  {\rm The second inequality of Proposition~\ref{Chudgen} can be thought of as an improvement of
Proposition~\ref{Chudprop} in the case in which $\beta(I^{(m)})< \alpha(I)m$. However, the proof actually shows that
$$\frac{\alpha(I^{(m)})}{m}\ge \Big(\frac{\alpha(I)+1}{2}\Big)\Bigg(\frac{\alpha(I)m}{\beta(J^{(m)})}\Bigg),\eqno{({}^*)}$$ 
where $J$ is as in the proof. Moreover,  $\beta(J^{(m)})\leq \alpha(I)m$, since $J$
is generated in degree $\alpha(J) = \alpha(I)$, and $J^m\subset J^{(m)}$. 
Thus $({}^*)$ does represent a small improvement on the orginal result of Chudnovsky. In fact,
suppose that equality occurs in Chudnovksy's bound, so that $\frac{\alpha(I^{(m)})}{m} =
 \frac{\alpha(I)+1}{2}$. Then necessarily $\alpha(I)m = \beta(J^{(m)})$ by $({}^*)$. But
$\alpha(I^{(m)}) \geq \alpha(J^{(m)})\geq (\frac{\alpha(J)+1}{2})m$, so that equality must
hold and we obtain that $\alpha(J^{(m)})\beta(J^{(m)}) = m^2\binom{\alpha(I)+1}{2}$.
Therefore by Proposition~\ref{Chudgen}, it follows that $(J^{(m)})^k = J^{(mk)}$ for all $k\geq 1$.}
\end{rem}

\begin{rem} {\rm The last conclusion in Theorem~\ref{Chudgen} implies that the 
symbolic power algebra, $\oplus I^{(n)}$ is a Noetherian ring. (This is a 
homogeneous version of \cite[Theorem 1.3]{refSc}.)} 
\end{rem}

\begin{defn}\label{starDef}\rm 
Let $H_1,\ldots,H_s\in \pr N$ be $s\geq N$ hyperplanes such that no $N+1$ meet at a 
single point, and let $p_1,\ldots,p_n$ be
the $n=\binom{s}{N}$ points such that each point is the intersection of a subset of $N$
of the $s$ hyperplanes; following a suggestion of Geramita, we refer to such a set of points $p_i$ as 
a {\it star configuration\/} for $s$ hyperplanes in $\pr N$ (since 5 general lines
in the plane if drawn appropriately give a 5 pointed star). 
\end{defn}

We now show that Conjecture \ref{fatptconj} and the conjecture of Chudnovsky mentioned in
Remark \ref{ChudConj} both hold for star configurations. In fact, we show more.

\begin{cor}\label{starcor1}
Let $I=\cap_i I(p_i)$ where the $n=\binom{s}{N}$ points $p_i\in\pr N$ give a 
star configuration coming from $s\ge N$ hyperplanes in $\pr N$.
Then $\frac{\alpha(I^{(r)})}{r}\geq\frac{\alpha(I)+N-1}{N}$ 
and $I^{(Nr)}\subseteq M^{r(N-1)}I^r$ hold for all $r\geq1$,
with equality in the former when $r$ is a multiple of $N$.
If moreover $N=2$ and $m>0$ is an even integer, then: 
equality holds in $({}^*)$ of Remark \ref{Chudimprove};
$\alpha(I^{(m)})\beta(I^{(m)})=m^2n$; and
$\alpha(I)k=\beta(I^{(k)})$ and $(I^{(m)})^k=I^{(mk)}$
hold for all $k\geq1$.
\end{cor}

\begin{proof}
We have $\alpha(I^{(Nr)})=sr$ by \cite[Lemma 2.4.1]{refBH}, so
$\gamma(I)=\lim_{r\to\infty}\alpha(I^{(rN)})/(rN)=s/N$,
but we also have ${\rm reg}(J)=\alpha(I)=s-N+1$ by \cite[Lemma 2.4.2]{refBH}, 
so we have the equality $\alpha(I^{(Nr)})/(rN)=\gamma(I)=s/N=(\alpha(I)+N-1)/N$,
as claimed. Since as pointed out above $\alpha(I^{(r)})/r\geq\gamma(I)$ 
holds for all $r$, we also have $\frac{\alpha(I^{(r)})}{r}\geq\frac{\alpha(I)+N-1}{N}$.

Since $\alpha(I)={\rm reg}(I)$, we see $I$ is generated in degree $\alpha(I)=s-(N-1)$
and since $\alpha(I^{(mrN)})=mrs\ge mr(s-(N-1))+mr(N-1)$, it follows 
by Proposition \ref{fatptprop1} that $I^{(Nr)}\subseteq M^{r(N-1)}I^r$.

Now assume $N=2$. 
Then we have $\frac{\alpha(I^{(m)})}{m}\geq\frac{\alpha(I)+1}{2}$
from above, so $({}^*)$ will be an equality if
we verify that $\beta(I^{(m)})= m\alpha(I)$.
Note that each of the $s$ lines defining the star configuration
contain exactly $s-1$ of the $n=\binom{s}{2}$ points $p_i$. 
Let $L$ be the linear form defining one of these lines.
If $F\in (I^{(k)})_t$ is a form of degree $t<k(s-1)$,
then $L$ divides $F$ by B\'ezout's Theorem. Thus $k(s-1)\leq \beta(I^{(k)})$,
but $\beta(I^{(k)})\leq {\rm reg}(I^{(k)})$, and,
by \cite[Theorem 1.1]{refGGP},  ${\rm reg}(I^{(k)})\leq k{\rm reg}(I)$. 
Since ${\rm reg}(I)=s-1$, we see $\beta(I^{(k)})= k(s-1)=k\alpha(I)$,
as we wanted to show. 
Moreover, $\alpha(I^{(m)})\beta(I^{(m)})=(sm/2)m(s-1)=m^2n$,
so $(I^{(m)})^k=I^{(mk)}$ holds by Proposition \ref{Chudgen}.
\end{proof}

Of course, star configurations are very special sets of points,
but if one takes the ideal $J$ of a general set of $\binom{s}{2}$ points of $\pr 2$,
then $\alpha(J)={\rm reg}(J)$ so $J$ is generated in degree $\alpha(J)$ and Conjecture \ref{fatptconj}
holds for $I=J^{(m)}$ by Proposition \ref{fatptprop}.
More generally, we now show that Conjecture \ref{fatptconj} holds for the radical ideal $I$ 
of any set of $n$ general points of $\pr2$.

\begin{prop}\label{genptsprop}
Let $I=\cap_i I(p_i)\subset R$ for $n$ general points $p_i\in\pr2$. 
Then $I^{(2r)}\subseteq M^rI^r$ holds for all $r$. 
\end{prop}

\begin{proof}
For $n=1,3,6$, $n$ is a binomial coefficient, so
$I^{(2r)}\subseteq M^rI^r$ holds as we observed immediately above. 
For $n=2, 4$, the points are a complete intersection,
and so $I^{(2r)}=I^{2r}\subseteq M^rI^r$ holds. 
By Proposition \ref{fatptprop1} it is enough to show that 
$I$ is generated in degrees $s$ and less for some $s$
such that $\alpha(I^{(2r)})/(2r)\ge (s+1)/2$ for all $r$.
Thus it is enough to show that $\gamma(I)\ge (s+1)/2$.

Consider the case $n=5$. By \cite[Lemma 3.1]{refBH2}, $\alpha(I^{(2r)})=4r$
and hence $\gamma(I)=2$. Since the 5 points impose independent conditions on
forms of degree $s$ for any $s\ge2$, we see ${\rm reg}(I)=3$ (so $I$ is generated
in degree $s=3$ and less). Thus $\gamma(I)=2\ge(s+1)/2$
as we wanted to show.

For $n=7$, $I$ is generated in degrees 3 and less \cite{refHa1} and $\gamma(I)=21/8$
(see the proof of \cite[Proposition 4.3]{refBH2}); 
for $n=8,9$, $I$ is generated in degrees 4 and less since ${\rm reg}(I)=4$, $\gamma(I)=48/17$
(see the proof of \cite[Proposition 4.4]{refBH2}) when $n=8$ and 
(it is easy to see) $\gamma(I)=3$ when $n=9$. Thus 
$I^{(2r)}\subseteq M^rI^r$ holds for $7\le n\le 9$.
Now say $n>9$. If $n$ is a binomial coefficient $\binom{s}{2}$ we saw above
that $I^{(2r)}\subseteq M^rI^r$ holds, so assume that
$\binom{s}{2}<n<\binom{s+1}{2}$ for some $s\ge 5$. 
It is known that $\gamma(I)\ge \sqrt{n-1}$ (see \cite[Remark 8.3.5]{refPSC}
and the proof of \cite[Theorem 4.2]{refBH})
and that $I$ is generated in degree at most
$s$ (since ${\rm reg}(I)=s$). So we want to check that $\sqrt{n-1}\ge (s+1)/2$, or that
$n-1\ge (s+1)^2/4$, but $n-1\ge \binom{s}{2}$ and $\binom{s}{2}>(s+1)^2/4$ for $s\ge 5$.
\end{proof}

For later use we have the following results regarding the ideal
of 5 general points of $\pr2$.

\begin{lem}\label{5genpts} Let $I$ be the ideal of 5 general points of $\pr2$.
Then $I^{(2r)}=(I^{(2)})^r$ and $I^{(2r+1)}=I^{(2r)}I$ for all $r\geq 1$.
\end{lem}

\begin{proof}
We saw in the proof of Proposition \ref{genptsprop} that $\alpha(I^{(m)})=2m$.
Also, by B\'ezout's Theorem we must have $2\beta(I^{(m)})\geq 5m$, and
by \cite[Remark I.5.5]{refHa2} (or by \cite[Theorem III.1(a)]{refHa3}), 
$(I^{(m)})_t$ has no non-constant common factors for $2t\ge 5m$. 
Thus $\beta(I^{(m)})=\lceil\frac{5m}{2}\rceil$. In particular,
$\alpha(I^{(2r)})\beta(I^{(2r)})=5(2r)^2$, so $I^{(2r)}=(I^{(2)})^r$
by Proposition \ref{Chudgen}.

Now consider $I^{(2r+1)}$. Clearly, $I^{(2r)}I=I^{2r}I=I^{2r+1}\subseteq I^{(2r+1)}$, so consider
the reverse inclusion.
By B\'ezout's Theorem, since $\beta(I^{(2)})=5$, we have $5\alpha(I^{(2r+1)})\geq 2*5(2r+1)$,
and hence $\alpha(I^{(2r+1)})\geq 2(2r+1)$, but $\alpha(I^{(2r+1)})\leq \alpha(I^{2r+1})=2(2r+1)$,
so we have $\alpha(I^{(2r+1)})=2(2r+1)$. Thus $0=(I^{(2r+1)})_t\subseteq (I^{(2r)}I)_t$ for $t<2(2r+1)$.
If $2(2r+1)\leq t < 5(2r+1)/2$, then $Q$ is a common factor for $(I^{(2r+1)})_t$, by B\'ezout's Theorem,
where $Q$ is a homogeneous form of degree 2 defining the unique conic through the five general points. 
Thus $(I^{(2r+1)})_t=Q(I^{(2r)})_{t-2}$, but $Q\in I_2$, so 
$(I^{(2r+1)})_t\subseteq(I^{(2r)})_{t-2}I_2\subseteq ((I^{(2r)})I)_t$.
Finally, assume $t \geq 5(2r+1)/2=5r+3$. Then 
$(I^{(2r+1)})_t=(I^{(2r)})_{5r}I_{3+(t-5r)}\subseteq ((I^{(2r)})I)_t$
by \cite[Proposition 2.4]{refBH2}, and hence 
$I^{(2r+1)}\subseteq I^{(2r)}I$.
\end{proof}

\begin{rem}\rm
We close this section with a discussion of issues raised by Proposition \ref{Chudgen}.
For example, the bound in Proposition \ref{Chudgen}(ii) improves on Chudnovsky's bound
in Proposition \ref{Chudprop} only when $m\alpha(I)>\beta(I^{(m)})$, but it certainly can happen that
$m\alpha(I)<\beta(I^{(m)})$. For example, consider the ideal $I$ of $n>d^2$ points on a smooth plane
curve of degree $d$. Then (by B\'ezout's Theorem) we have both $\alpha(I)=d$ and
$\beta(I^{(m)})d\geq nm$, so $\beta(I^{(m)})\geq nm/d > \alpha(I)m$.
There also are cases with $m\alpha(I)>\beta(I^{(m)})$. For example, given $m>0$,
let $a=(2m+1)^2-1$ and let $n=2(2m+1)$. It is easy now to check that
$\binom{a+1}{2}=n^2\binom{m+1}{2}$. Let $I$ be the ideal of the union $Z$ of
$n^2$ general points in $\pr2$. 
By \cite{refCM,refEv,refR}, the fat point scheme $iZ$ imposes
independent conditions on forms of degree $t$ for any $t$ such that $(I^{(i)})_t\neq0$.
In particular, it follows that $\binom{\alpha(I)+2}{2}>n^2$, so $m\alpha(I)>m(\sqrt{2}n-2)$.
Since $\binom{a+1}{2}=n^2\binom{m+1}{2}$, it also follows that
$\beta(I^{(m)})\leq {\rm reg}(I^{(m)})=a=(2m+1)^2-1$, hence for $m\gg0$ we have
$\beta(I^{(m)})< (2m+1)^2< 4\sqrt{2}m^2<m(\sqrt{2}n-2)<m\alpha(I)$.

Another issue raised by Proposition \ref{Chudgen} is whether there are cases 
of $n$ points in the plane and integers $m>0$ where the ideal $I$ of the points satisfies
$\alpha(I^{(m)})\beta(I^{(m)})=m^2n$. 
This holds for the ideal $I$ of a star configuration by Corollary \ref{starcor1},
hence $I^{(2k)}=(I^{(2)})^k$ for all $k\geq1$.
For additional examples, consider the scheme $Z$ consisting of
$n$ general points of the plane. For $n\leq8$ points, we can assume the points
lie on a smooth cubic curve, in which case the results of 
\cite{refHa3, refGuH, refGHM} can be used to determine $\alpha(I^{(m)})$ and $\beta(I^{(m)})$.
For $n=1$, $\alpha(I^{(m)})=\beta(I^{(m)})=m$ for all $m\geq1$, hence
by Proposition \ref{Chudgen} we have $I^{(k)}=I^k$ for all $k\geq1$. Of course,
in this case $Z$ is a complete intersection.
For $n=2$, $\alpha(I^{(m)})=m$ and $\beta(I^{(m)})=2m$ for all $m\geq1$, hence
we have $I^{(k)}=I^k$ for all $k\geq1$. Of course,
in this case $Z$ is again a complete intersection.
For $n=3$, $\alpha(I^{(m)})=\lceil 3m/2\rceil$ and $\beta(I^{(m)})=2m$ for all $m\geq1$, hence
we have $I^{(2k)}=(I^{(2)})^k$ for all $k\geq1$. In this case $Z$ is a star configuration.
For $n=4$, $\alpha(I^{(m)})=\beta(I^{(m)})=2m$ for all $m\geq1$, hence
we have $I^{(k)}=I^k$ for all $k\geq1$. Of course,
in this case $Z$ is yet again a complete intersection.
For $n=5$, we have $I^{(2k)}=(I^{(2)})^k$ for all $k\geq1$; see Lemma \ref{5genpts} and its proof. 
For $n=6$, $\alpha(I^{(m)})=\lceil 12m/5\rceil$ and $\beta(I^{(m)})=\lceil 5m/2\rceil$ for all $m\geq1$, hence
we have $I^{(10k)}=(I^{(10)})^k$ for all $k\geq1$. 
For $n=7$, $\alpha(I^{(m)})=\lceil 21m/8\rceil$ and $\beta(I^{(m)})=\lceil 8m/3\rceil$ for all $m\geq1$, hence
we have $I^{(24k)}=(I^{(24)})^k$ for all $k\geq1$. 
For $n=8$, $\alpha(I^{(m)})=\lceil 48m/17\rceil$ and $\beta(I^{(m)})=\lceil 17m/6\rceil$ for all $m\geq1$, hence
we have $I^{(102k)}=(I^{(102)})^k$ for all $k\geq1$. 

Again applying the results of \cite{refHa3},
for $n=9$ general points (which we may assume therefore lie on a smooth cubic curve), 
we have $\alpha(I^{(m)})=3m$ and $\beta(I^{(m)})=3m+1$,
so we never have $\alpha(I^{(m)})\beta(I^{(m)})=m^2n$, but if the 9 points $p_i$ are chosen
to be points on a smooth plane cubic curve such that
$p_1+\cdots+p_9-p$ has order $r$ in the divisor class group of the cubic (where $p$ is a flex point), 
then $\alpha(I^{(m)})=3m$ and $\beta(I^{(m)})$ is $3m$ if $r|m$, and $3m+1$ if $r$ does not divide $m$.
Thus $\alpha(I^{(r)})\beta(I^{(r)})=9r^2$, so $I^{(rk)}=(I^{(r)})^k$ for all $k\geq0$.
\end{rem}

\begin{rem}\rm
Finally, we raise two questions. Given distinct points $p_i$ in the plane and positive integers $m_i$,
consider the ideal $I=\cap_i\,I(p_i)^{m_i}$. Is it true that $I^{(mk)}=(I^{(m)})^k$ for $k\geq1$
(and hence that the symbolic Rees algebra $\oplus I^{(n)}$ is Noetherian)
if $\alpha(I^{(m)})\beta(I^{(m)})=m^2\sum_im_i^2$?
See \cite[Example 5.1]{refBH2} for examples with $\alpha(I)\beta(I)=\sum_im_i^2$
for which $I^{(k)}=I^k$ for all $k\geq1$. Conversely, if $\oplus I^{(n)}$ is Noetherian,
must $\alpha(I^{(m)})\beta(I^{(m)})=m^2\sum_im_i^2$ hold for some $m$?
\end{rem}

\renewcommand{\thethm}{\thesubsection.\arabic{thm}}
\setcounter{thm}{0}

\section{Additional questions and conjectures}\label{section: addquests}

If for some ideal $I$ there is a $d$ such that one has $I^{(m)}\subseteq I^r$ for all
$m\geq dr$, one can next ask for what constants $c$ does
$m\geq dr-c$ guarantee $I^{(m)}\subseteq I^r$. We discuss questions and conjectures
related to this problem of subtracting a constant in Section \ref{sect4.1}. In Section \ref{sect4.2}
we discuss questions and conjectures arising out of refinements of the
Waldschmidt-Skoda bound $({}^\circ)$ of the Introduction.

\subsection{Subtracting a constant}\label{sect4.1}
The second named author has raised the question whether a radical ideal $I$ of a finite set of points 
in $\pr2$ always satisfies $I^{(3)}\subseteq I^2$. Examples suggested to the first named author the following
conjectural generalization \cite[Conjecture 8.4.2]{refPSC}:

\begin{conj}\label{Essenconj}
Let $I\subseteq K[\pr N]$ be a homogeneous ideal.
Then $I^{(rN-(N-1))}\subseteq I^r$ holds for all $r$. 
\end{conj}

Conjecture \ref{Essenconj} holds for star configurations \cite[Example 8.4.8]{refPSC};
examples of star configurations also show $I^{(rN-N)}\subseteq I^r$ fails in general.
When $N=2$, Conjecture \ref{Essenconj} is true 
for any finite set of general points in $\pr2$ by \cite[Remark 4.3]{refBH}.
Thus, while Conjecture \ref{Essenconj} is open in general, it is plausible (at least for ideals of points).
In fact, the second named author observed that
the principle underlying the main result of \cite{refHH} shows that Conjecture \ref{Essenconj} 
is true for radical ideals of any finite set of points if ${\rm char}(K)=p>0$ when $r$ is a power of $p$
\cite[Remark 8.4.4]{refPSC}.
The same basic argument also verifies Conjecture \ref{Essenconj} for monomial ideals 
(with no restriction on the characteristic) \cite[Remark 8.4.5]{refPSC}. 

The key in both cases is the use of nice behavior of Frobenius powers.
If $I\subseteq R$ is an ideal, define its $q$th {\it Frobenius power\/} $I^{[q]}$ to be the ideal generated by
all $v^q$ for $v\in I$. If $I$ is a monomial ideal, then $I^{[q]}$ is generated by the $q$th
powers of any set of monomial generators of $I$. And if ${\rm char}(K)=p>0$ and $q$ is a power of $p$,
then $I^{[q]}$ is generated by $q$th powers of any set of generators of $I$.
A fundamental fact for ideals $J_1,\ldots,J_s\subseteq R$ is that 
$(\cap_i J_i)^{[q]}=\cap_i (J_i^{[q]})$ if either $J_i$ is monomial for each $i$ (see \cite[Remark 8.4.5]{refPSC})
or if $p={\rm char}(K)>0$ and $q$ is a power of $p$ (by flatness of Frobenius; see \cite[Lemma 13.1.3, p. 247]{refSH} 
and \cite{refK}). As a direct consequence we  obtain:

\begin{lem}\label{FrobPowerProp}
Let $g>0$ be an integer and let
$J_1,\ldots,J_s\subseteq R$ be ideals, each generated by at most $g$ elements. 
Assume either that each $J_i$ is monomial, or that $p={\rm char}(K)>0$ and $r$ is a power of $p$.
Then $\cap_i (J_i^m)\subseteq (\cap_i J_i)^r$ as long as $m\ge gr-g+1$.
\end{lem}

\begin{proof}
Because $J_i$ has at most $g$ generators, any product of a choice of $gr-g+1$ 
of these generators is divisible by the $r$th power of one of the generators.
Thus $J_i^m\subseteq J_i^{gr-g+1}\subseteq J_i^{[r]}$, so we have 
$\cap_i (J_i^m)\subseteq \cap_i (J_i^{[r]})=(\cap_iJ_i)^{[r]}\subseteq (\cap_i J_i)^r$.
\end{proof}

\begin{rem} \rm We now see that Conjecture \ref{Essenconj} holds for any monomial ideal $I\subset K[\pr N]$.
If $I$ is not saturated, then $I^{(rN-(N-1))}=I^{rN-(N-1)}\subseteq I^r$, and if $I$ is saturated
then there is a (not necessarily irredundant) primary decomposition 
$I=\cap_i J_i$ where each $J_i$ is monomial, primary and generated by
positive powers of the variables in some proper subset (depending on $i$)
of the $N+1$ variables (and hence $J_i$ has at most $N$ generators). 
Moreover, it is not hard to show that $I^{(m)}\subseteq \cap_i (J_i^m)$ (see  \cite[Remark 8.4.5]{refPSC}).
Lemma \ref{FrobPowerProp} now applies, and we have $I^{(m)}\subseteq I^r$ whenever
$m\ge rN-N+1$. Similarly, if $p={\rm char}(K)>0$ and $r$ is a power of $p$, then
for the ideal $I=\cap_i I(p_i)$ of any finite set of distinct points $p_i\in\pr N$,
by Lemma \ref{FrobPowerProp} we have 
$I^{(m)}=\cap_i (I(p_i)^m)\subseteq(\cap_i I(p_i))^r=I^r$ 
as long as $m\ge rN-N+1$.

\end{rem}

When $N=2$ Chudnovsky's bound
$\gamma(I)\ge (\alpha(I)+1)/2$ suggests another conjecture (in the perhaps weak sense that
Conjecture \ref{p2conj} implies the bound):

\begin{conj}\label{p2conj}
Let $I\subseteq K[\pr 2]$ be the homogeneous radical ideal of a finite set of points.
Then $I^{(m)}\subseteq I^r$ holds whenever $m/r\ge 2\alpha(I)/(\alpha(I)+1)$. 
\end{conj}

This conjecture does indeed imply Chudnovsky's bound
$\gamma(I)\ge (\alpha(I)+1)/2$. To see this, let $m=2\alpha(I)t$ and let $r=(\alpha(I)+1)t$
for integers $t\geq1$. Conjecture \ref{p2conj} then implies that $I^{(m)}\subseteq I^r$ which in turn implies
$\alpha(I^{(m)})\geq\alpha(I^r)$ and hence $\alpha(I^{(m)})/m\geq (r/m)\alpha(I)$.
Taking limits as $t\to\infty$ gives $\gamma(I)\geq (\alpha(I)+1)/2$.

Conjecture \ref{p2conj} is asymptotically stronger than Conjecture \ref{Essenconj} since 
$2r-1> 2r\alpha(I)/(\alpha(I)+1)$ for $r\gg0$.
But as with Conjecture \ref{Essenconj}, Conjecture \ref{p2conj} holds
for star configurations: if $I$ is the ideal of a star configuration,
then $\alpha(I)={\rm reg}(I)$ and $\gamma(I)=(\alpha(I)+1)/2$, hence we have
$I^{(m)}\subseteq I^r$ by \cite[Theorem 1.2.1(b)]{refBH} (since $m/r\ge 2\alpha(I)/(\alpha(I)+1)
= {\rm reg}(I)/\gamma(I)$).

It is easy to see that Conjecture \ref{p2conj} holds if $I$ is a complete 
intersection, since then $I^{(m)}=I^m$.
Conjecture \ref{p2conj} also holds 
for general sets of $n$ points (use \cite[Remark 4.3]{refBH} 
in case $n\ge6$, since then
$\alpha(I)\ge3$ and so $2\alpha(I)/(\alpha(I)+1)\ge 3/2$, use \cite[Theorem 3.4(b)]{refBH2} for
$n=5$, and for $n<5$ note that $n$ general 
points give either a star or a complete intersection or both). 

Giving Conjecture \ref{Essenconj} an evolutionary twist, we obtain another possibility:

\begin{conj}\label{EvoEssenconj}
Let $I\subseteq K[\pr N]$ be the ideal of a finite set of points $p_i\in\pr N$.
Then $I^{(rN-(N-1))}\subseteq M^{(r-1)(N-1)}I^r$ holds for all $r\geq1$. 
\end{conj}

\begin{lem}\label{EvoEssenconjLem}
Let $I\subseteq K[\pr N]$ be the ideal of a finite set of points $p_i\in\pr N$.
If $I^{(Nr-(N-1))}\subseteq I^r$ and $\alpha(I^{(Nr-(N-1))}) \ge rs + (r-1)(N-1)$
for some $s$ such that $I$ is generated in degrees $s$ and less, then
$I^{(rN-(N-1))}\subseteq M^{(r-1)(N-1)}I^r$. 
\end{lem}

\begin{proof}
The same argument (with $m=1$) used in the proof of Proposition \ref{fatptprop1}
works here.
\end{proof}

\begin{cor}\label{EvoEssenconjRem}
Let $I\subseteq K[\pr N]$ be the ideal of a finite set of points $p_i\in\pr N$,
comprising either a star configuration or a complete intersection.
Then $I^{(rN-(N-1))}\subseteq M^{(r-1)(N-1)}I^r$ holds for all $r\geq1$. 
 \end{cor}

\begin{proof}
Since for star configurations $\alpha(I^{(Nr-(N-1))})=r{\rm reg}(I) + (r-1)(N-1)$ 
by \cite[Lemma 8.4.7]{refPSC},
Lemma \ref{EvoEssenconjLem} applies with $s={\rm reg}(I)$
both conditions of Lemma \ref{EvoEssenconjLem}, so the result follows.
For complete intersections we have $I^{(rN-(N-1))}=I^{rN-(N-1)}
= I^{rN-r-(N-1)}I^r=I^{(r-1)(N-1)}I^r\subseteq M^{(r-1)(N-1)}I^r$.
\end{proof}

Of course, if Conjecture \ref{EvoEssenconj} is true, then so must be the following:

\begin{conj}\label{EvoEssenconj2}
Let $I\subseteq K[\pr N]$ be the ideal of a finite set of points $p_i\in\pr N$.
Then 
$$\alpha(I^{(rN-(N-1))})\ge r\alpha(I)+(r-1)(N-1)$$ 
for every $r>0$.
\end{conj}

\begin{rem}\label{ETSr=1}\rm
Conjecture \ref{EvoEssenconj2} is clearly true if $\alpha(I)=1$ or if $r=1$.
It also holds if $\gamma(I)\ge (r\alpha(I)+(r-1)(N-1))/(rN-(N-1))$, since 
$\alpha(I^{(rN-(N-1))})/(rN-(N-1))\ge \gamma(I)$.
But if $\alpha(I)>1$, then $(r\alpha(I)+(r-1)(N-1))/(rN-(N-1))$ is biggest when $r$ is least,
so Conjecture \ref{EvoEssenconj2} will hold for all $r$
if $\gamma(I)\ge (r\alpha(I)+(r-1)(N-1))/(rN-(N-1))$ holds for $r=2$. 
\end{rem}

\begin{rem}\rm
As mentioned in Remark \ref{ChudConj}, if $I$ is the radical ideal of a finite set of 
points $p_i\in\pr N$, Chudnovsky conjectured
that $\gamma(I)\ge (\alpha(I)+N-1)/N$. It seems plausible in fact that $\gamma(I)>(\alpha(I)+N-1)/N$
unless either the points lie on a hyperplane or give a star configuration. 
This is however open, even for $N=2$. The best current result is:
if $N=2$ and if $\alpha(I^{(2)})/2=(\alpha(I)+1)/2$, then the points $p_i$ either lie on
a line or give a star configuration \cite{refBC}. If it were true that
$\gamma(I)>(\alpha(I)+N-1)/N$ unless either the points were contained in a hyperplane
or formed a star configuration, then Conjecture \ref{EvoEssenconj2} would at least hold for all
$r\gg0$, as the next result shows.
\end{rem}

\begin{cor}\label{starcor}
For the radical ideal $I$ of  a finite set of points $p_1,\ldots,p_n\in\pr N$ we have 
$$\alpha(I^{(Nr-N+1)})\ge r\alpha(I)+(r-1)(N-1)$$
for all $r\gg0$ 
if either $\alpha(I)=1$ or the points $p_i$ form a star configuration or $\gamma(I)>(\alpha(I)+N-1)/N$.
\end{cor}

\begin{proof}
If $\alpha(I)=1$, then $\alpha(I^{(rN-(N-1))})=rN-(N-1)=r\alpha(I)+(r-1)(N-1)$.
By \cite[Lemma 8.4.7 ]{refPSC}, if the points form a star configuration on $s$ lines, then 
$\alpha(I)=s-N+1$ and $\alpha(I^{(rN-(N-1))})=(r-1)s+s-N+1=r\alpha(I)+(r-1)(N-1)$.
If $\gamma(I)>(\alpha(I)+N-1)/N$, then $\gamma(I)=\delta+(g+N-1)/N$
for some $\delta>0$ where $g=\alpha(I)$. Hence for $r\gg0$ we have
$\alpha(I^{(Nr-(N-1))})\ge rg+(r-1)(N-1)$, since
$\alpha(I^{(Nr-(N-1))})\ge (Nr-(N-1))\gamma(I) = (Nr-(N-1))\delta+(Nr-(N-1))(g+N-1)/N
=rg+(r-1)(N-1)+(N(Nr-(N-1))\delta-(g-1)(N-1))/N$ and the last term is positive for $r\gg0$.
\end{proof}

\begin{example}\label{genptsexample1}\rm
We now check that Conjecture \ref{EvoEssenconj2} holds for
every set of $n\le8$ points of $\pr2$. By Remark \ref{ETSr=1}, it is enough to check 
$\gamma(I)\ge (2\alpha(I)+1)/3$.
When $N=2$, $\gamma(I)$ can be found for each set of $n\le8$ points of $\pr2$
using the results of \cite{refGuH, refGHM} \newline
(see \url{http://www.math.unl.edu/~bharbourne1/GammaFile.html}). It turns out that\newline
$\gamma(I)\ge (2\alpha(I)+1)/3$ holds for every configuration
of $n\le 8$ points except for four cases: the 3 points coming from
the star for 3 lines; the 6 points coming from the star
for 4 lines; 6 points where
3 of them are a star for 3 lines and an additional point is chosen on each of those 3 lines 
but such that these three additional points are not collinear; 
and 7 points where 6 of them form the star on 4 lines and an
additional point is placed on one of those 4 lines.
In each of these cases except for the star configuration of 6 points,
$\gamma(I)\ge (r\alpha(I)+r-1)/(2r-1)$ holds for $r=3$ (and hence by Remark \ref{ETSr=1} for $r\ge3$)
and $\alpha(I^{(2r-1)})\ge (r\alpha(I)+r-1)/(2r-1)$ holds for $r=2$.
This leaves the 6 point star configuration, but Corollary \ref{starcor} shows
that Conjecture \ref{EvoEssenconj2} holds for stars.
\end{example}

Now we show Conjecture \ref{EvoEssenconj} holds in the case of $n$ general points when $N=2$.
(We use the characteristic 0 hypothesis only for some values of $r$ and $n$ where 
we apply \cite{refD}; see the last paragraph of the proof.)

\begin{cor}\label{genptsexample2}
Let $I$ be the ideal of $n$ general points of $\pr2$; assume the ground field $K$ has 
${\rm char}(K)=0$. Then $I^{(2r-1)}I\subseteq M^{r-1}I^r$.
\end{cor}

\begin{proof} As noted in Corollary \ref{EvoEssenconjRem}, 
$I^{(2r-1)}I\subseteq M^{r-1}I^r$ holds for complete intersections and for star configurations.
Thus it holds for $n=1,2,4$ since in these cases we have a
complete intersection, and it holds for $n=3$ since this is a star configuration.
It holds for $n=5$ by Proposition \ref{genptsprop} and Lemma \ref{5genpts},
since $I^{(2r-1)}=I^{(2(r-1))}I\subseteq M^{r-1}I^{r-1}I=M^{r-1}I^r$.
For larger $n$, note that $I^{(2r-1)}\subseteq I^r$ holds for general points by \cite[Example 8.4.9]{refPSC},
so $I^{(2r-1)}I\subseteq M^{r-1}I^r$ also holds  (by Lemma \ref{EvoEssenconjLem})
if $\gamma(I)\ge (r{\rm reg}(I)+r-1)/(2r-1)$. Therefore $I^{(2r-1)}I\subseteq M^{r-1}I^r$
holds for $n=6$ since then $\alpha(I)={\rm reg}(I)$ and 
in Example \ref{genptsexample1} we verified that 
$\gamma(I)\ge (r\alpha(I)+r-1)/(2r-1)$ holds for $r\ge2$ for every configuration
of $3\ne n\le 8$ general points.
Similarly, $I^{(2r-1)}I\subseteq M^{r-1}I^r$ holds for $n=7$ since (as noted in the proof of
Proposition \ref{genptsprop}) $I$ is generated
in degree $\alpha(I)$; now argue as in the case $n=6$. For $n=8$, ${\rm reg}(I)=4$
and (as in the proof of Proposition \ref{genptsprop})
we have $\gamma(I)=48/17\ge (r{\rm reg}(I)+r-1)/(2r-1)$
for $r=3$ (and hence for all $r\ge3$), while for $r=2$, we have 
$\alpha(I^{(2r-1)})=9=r{\rm reg}(I)+r-1$,
so $I^{(2r-1)}I\subseteq M^{r-1}I^r$ holds for $n=8$.
And for $n=9$ general points of $\pr2$, 
$\gamma(I)=3=(r{\rm reg}(I)+r-1)/(2r-1)$ for $r=2$, so again
$I^{(2r-1)}I\subseteq M^{r-1}I^r$ holds. 

Now let $n\ge 10$; as in the proof of Proposition \ref{genptsprop}, $\gamma(I)\ge \sqrt{n-1}$.
If $n=\binom{s+1}{2}$ for some $s\ge 4$, then ${\rm reg}(I)=\alpha(I)=s$,
and so we have $\sqrt{n-1}\ge (r{\rm reg}(I)+r-1)/(2r-1)$ for $r=2$,
hence for all $r$. So now assume
$\binom{s}{2}<n<\binom{s+1}{2}$ for some $s\ge 5$. 
Then ${\rm reg}(I)=s$, so we want to check that $\sqrt{n-1}\ge (2s+1)/3$, or that
$n-1\ge (2s+1)^2/9$, but $n-1\ge \binom{s}{2}$ and $\binom{s}{2}\ge(2s+1)^2/9$ for $s\ge 18$.
So we need to check $s\le 17$; i.e., $10<n<153$. By direct check we have
$\sqrt{n-1}\ge (r{\rm reg}(I)+r-1)/(2r-1)$ for $r=7$ for $10<n<153$.
So now we just need to check that 
$\alpha(I^{(2r-1)})\ge r\alpha(I)+(r-1)$ holds for $2\le r\le 6$ for $10<n<153$.
We verified this using \cite{refD} to determine $\alpha(I^{(2r-1)})$; note that
\cite{refD} assumes characteristic 0.
\end{proof}

\begin{rem}\rm
It may be worthwhile to consider the maximum height of the associated primes.
If $e$ is the maximum of the heights of the associated primes of a given homogeneous
ideal $I$, then from \cite{refHH} we know $I^{(re)}\subseteq I^r$, and it is conjectured
in \cite{refPSC} that $I^{(re-(e-1))}\subseteq I^r$. This raises the question of whether
$I^{(re)}\subseteq M^{r(e-1)}I^r$ and $I^{(re-(e-1))}\subseteq M^{(r-1)(e-1)}I^r$ are also true.
\end{rem}

\begin{rem}\rm
As a minor remark, we show how classical methods can be used to show
$I^{(2m)}\subseteq I^m$ in some cases of ideals of points $p_1,\ldots,p_n\in\pr2$.
Assume $I$ is a radical ideal for a finite set of points in $\pr2$. Assume the characteristic is 0, and that
$\alpha(I)={\rm reg}(I)$. Using the Euler identity as in the proof of Fact \ref{homogEMfact},
we have $I^{(2m)}\subseteq M^mI^{(m)}$. But $\alpha(I)={\rm reg}(I)$ implies
that $I^m=M^j\cap I^{(m)}$ for $j=m\alpha(I)=\alpha(I^m)$, since the saturation degree
of $I^m$ is bounded above by $m{\rm reg}(I)$ which by our hypothesis is equal to $\alpha(I^m)$,
so truncating $I^{(m)}$ at degree $\alpha(I^m)$ gives $I^m$ (see \cite[Remark 4.2]{refBH2}).

Since $M^mI^{(m)}\subseteq I^{(m)}$, we now see that
$M^mI^{(m)}\subseteq M^j\cap I^{(m)}=I^m$ if $M^mI^{(m)}\subseteq M^j$; i.e., if
$\alpha(M^mI^{(m)})\ge j$. But $\alpha(M^mI^{(m)})=\alpha(I^{(m)})+m\ge m\gamma(I)+m$
so $\alpha(M^mI^{(m)})\ge j$ holds if $m\gamma(I)+m\ge j=m\alpha(I)$; i.e., if
$\gamma(I)\ge \alpha(I)-1$. This holds if $(\alpha(I)+1)/2\ge \alpha(I)-1$ (i.e., if $\alpha(I)\le 3$)
since $\gamma(I)\ge (\alpha(I)+1)/2$.

This shows that $I^{(2m)}\subseteq I^m$ holds in characteristic 0
if $\alpha(I)={\rm reg}(I)$ and $3\ge \alpha(I)$. We note that
$\alpha(I)={\rm reg}(I)$ implies that $n$ is a binomial coefficient, so
we have either $n=6$ and $p_1,\ldots,p_6$
lie on a cubic but not on a conic (there are five essentially different such configurations
\cite{refGuH}), or $n=3$ and $p_1,p_2,p_3$ lie on a conic but not on a line,
or $n=1$.  

Additional cases follow from $\gamma(I)\ge \alpha(I)-1$. For example, for
any 10 points which do not lie on a 
cubic, we have $\alpha(I)={\rm reg}(I)$. If some 9 of the points lie on a smooth cubic,
then $3\alpha(I^{(m)})\ge 9m$ holds by B\'ezout's Theorem, 
so we have  $\gamma(I)\ge 3= \alpha(I)-1$ and thus
$I^{(2m)}\subseteq I^m$. 
\end{rem}

\renewcommand{\thethm}{\thesubsection.\arabic{thm}}
\setcounter{thm}{0}

\subsection{Further refinements}\label{sect4.2}

A refinement of $({}^\circ)$ of the Introduction is given in \cite[Lemme 7.5.2]{refW2}.
In our terms, this refinement is that
$$\frac{\alpha(I^{(m)})}{m+N-1}\leq \gamma(I).\eqno{({}^{**})}$$
for the radical ideal $I$ for a finite set of points in the complex projective space $\pr N$.
The proof given in \cite{refW2} uses complex analytic techniques;
for an easy proof using multiplier ideals, see 
\cite[Proposition 10.1.1 and Example 10.1.3]{refLa}.

In fact, by a variation of the proof given in the introduction for the case $m=1$, 
$({}^{**})$ holds for any homogeneous ideal $0\neq I\subseteq K[\pr N]$ over any field $K$. In particular, 
we have $I^{(t(m+N-1))}\subseteq (I^{(m)})^t$ for $m\geq 1$ by \cite{refHH1},
so $t\alpha(I^{(m)})\leq \alpha(I^{(t(m+N-1))})$. Dividing by $t(m+N-1)$
and taking limits as $t\to\infty$ gives $({}^{**})$. 

Again let $I$ be the radical ideal $I$ for a finite set of points in the projective space $\pr N$
but over the complex numbers.
A further refinement, proved using complex projective techniques, is given in \cite{refEV}
for the case that $N\geq 2$:
$$\frac{\alpha(I^{(m)})+1}{m+N-1}\leq \gamma(I).\eqno{({}^{***})}$$
This is just Proposition \ref{Chudprop} when $m=1$ and $N=2$.
Comparing $({}^{***})$ with Chudnovsky's conjecture 
$(\alpha(I)+N-1)/N\leq \gamma(I)$ raises the question: 

\begin{ques}\label{genChudques}
Let $I$ be the radical ideal $I$ for a finite set of points in $\pr N$.
Is it true for all $m\geq 1$ that 
$$\frac{\alpha(I^{(m)})+N-1}{m+N-1}\leq \gamma(I)?$$
\end{ques}

Using the fact that for the ideal $I$ of a star configuration defined by $s$ hyperplanes in $\pr N$
with $m=Ni+j$ for $0\leq i$ and $0<i\leq N$ we have $\alpha(I^{(m)})=(i+1)s-N+j$,
one can check that Question \ref{genChudques} has an affirmative answer for star configurations.

These speculations and observations raise two additional questions for the radical ideal
$I$ of a finite set of points in $\pr N$:

\begin{ques}\label{genChudques2}
Is it true for all positive integers $m$ and $t$ that $I^{(t(m+N-1))}\subseteq M^t(I^{(m)})^t$,
where $M\subset K[\pr N]$ is the ideal generated by the variables?
\end{ques}

If Question \ref{genChudques2} has an affirmative answer, 
then we obtain an alternate proof of $({}^{***})$ in the usual way.

Finally we ask:

\begin{ques}\label{genChudques3}
Is it true for all positive integers $m$ and $t$ 
that $I^{(t(m+N-1))}\subseteq M^{t(N-1)}(I^{(m)})^t$?
\end{ques}

If so, then Question \ref{genChudques} must also have an affirmative answer.

\end{document}